\documentclass[12pt]{gtpart}
\usepackage{amssymb,amsfonts,amsthm,amsmath,amscd}
\usepackage{url}
\title{Strong embeddability and extensions of groups}

\author{Ronghui Ji}
\address{
		Department of Mathematical Sciences\\
		Indiana University-Purdue University, Indianapolis\\ 
		402 N. Blackford Street\\
		Indianapolis, IN 46202\\
		}
\email{ronji@math.iupui.edu}

\author{Crichton Ogle}
\address{
		Department of Mathematics\\
		The Ohio State University\\
        Columbus, OH 43210\\
        }
\email{ogle@math.ohio-state.edu}

\author{Bobby W. Ramsey}
\address{
		Department of Mathematics\\
		The Ohio State University\\
        Columbus, OH 43210\\
        }
\email{ramsey.313@math.osu.edu}

\newtheorem*{thm*}{Theorem}
\newtheorem*{cor*}{Corollary}
\newtheorem{thm}{Theorem}[section]
\newtheorem{defn}[thm]{Definition}
\newtheorem{prop}[thm]{Proposition}
\newtheorem{cor}[thm]{Corollary}
\newtheorem{lem}[thm]{Lemma}

\newcommand{\N}{{\mathbb{N}}}

\newcommand{\isom}{{\,\cong\,}}

\newcommand{\mH}{\mathcal{H}}
\newcommand{\I}{\mathcal{I}}
\newcommand{\Y}{\mathcal{Y}}

\newcommand{\supp}{\operatorname{supp}}

\allowdisplaybreaks[4]

\begin{document}

\begin{abstract}
We introduce the notion of strong embeddability for a metric space.  This property lies between coarse embeddability and property A.
A relative version of strong embeddability is developed in terms of a family of set maps on the metric space.  When restricted to
discrete groups, this yields relative coarse embeddability.  We verify that groups
acting on a metric space which is strongly embeddable has this relative strong embeddability, provided the stabilizer subgroups do.
As a corollary, strong embeddability is preserved under group extensions.
\end{abstract}   

\maketitle

\section{Introduction}
If a discrete group $G$ has property A, as defined by Yu in \cite{yu}, it may be coarsely embedded in a separable Hilbert space. 
This is a desirable property, as Yu has shown \cite{yu} that such coarse embeddability verifies the Coarse Baum-Connes Conjecture 
for that group. In \cite{DG_CPUE}, it was shown that if $H \to G \to Q$ is an extension of groups where $Q$ has property A and $H$ 
is coarsely embeddable, then $G$ is coarsely embeddable. Unlike property A, it is unknown whether the property of coarse embeddability 
is closed under arbitrary extensions.  Moreover, it is known for metric spaces that property A is a strictly stronger condition than 
coarse embeddability \cite{AGS}.

In this paper we introduce an intermediate notion of strong coarse embeddability implied by property A and implying coarse embeddability.
This is done via a careful analysis of the results of Dadarlat and Guentner in \cite{DG_CPUE}.
A main property of this stronger notion is that it is preserved under arbitrary extensions.  More generally, strong embeddability admits a 
natural relative formulation by which one can verify a relative generalization of the extension
theorem.  Namely:  If $H < G$ is strongly embeddable and $\{ \pi : G \to G/H \}$ is a \emph{strongly embeddable family of maps}
then $G$ is strongly embeddable.  This strong embeddability of a family of maps is a natural generalization of relative property A
studied in \cite{JOR4}.  As a consequence, if $G$ has relative property A with respect to a family $H_1, \ldots, H_n$ of subgroups,
and if each $H_i$ is coarsely embeddable, then $G$ is coarsely embeddable.
Finally, we show that strong embeddability is satisfied by groups acting on strongly embeddable metric spaces.

It is currently unknown whether strong embeddability and coarse embeddability coincide for discrete metric spaces.  We hope to address
this issue in future work.

\section{Preliminaries}
All groups are assumed to be countable and uniformly discrete, with a proper left-invariant metric.  All metric
spaces are assumed to be uniformly discrete with bounded geometry.

\begin{defn}
A \emph{coarse embedding} of a metric space $(X,d)$ in a Hilbert space $\mH$ is a map $\phi : X \to \mH$ for which
there exist two nondecreasing maps $\rho_{-}, \rho_{+} : [0, \infty) \to \R$ with $\lim_{r\to\infty} \rho_{\pm}(r) = \infty$
for which the following holds for all $x, y \in X$.
\[ \rho_{-}\left(  d(x,y) \right) \leq \| \phi(x) - \phi(y) \| \leq \rho_{+}\left( d(x,y) \right). \]
A metric space for which a coarse embedding exists is called \emph{coarsely embeddable}.
\end{defn}

A useful characterization of coarse embeddability was given by Dadarlat and Guentner in \cite{DG_CPUE}.
\begin{prop}\label{prop:CE}
Let $X$ be a metric space.  Then $X$ is coarsely embeddable if and only if for every $R, \epsilon > 0$ there
exists a Hilbert space valued map $\beta : X \to \mH$ with $\| \beta(x) \| = 1$ for each $x \in X$, and satisfying:
\begin{enumerate}
	\item $\sup \left\{ \left| 1 - \langle \beta(x), \beta(y) \rangle \right| : d(x,y) \leq R \right\} \leq \epsilon$.
	\item $\lim_{S \to \infty} \sup \left\{ \left| \langle \beta(x) , \beta(y) \rangle \right| : d(x,y) \geq S \right\} = 0$.
\end{enumerate}	
\end{prop}
We point out that condition (1) can be replaced by 
	$\sup \left\{ \| \beta(x) - \beta(y) \| : d(x,y) \leq R \right\} \leq \epsilon$. 

We will need to make use of a family of coarsely embeddable metric spaces, with some uniform control.  We take the following
as our definition.

\begin{defn}[{\cite[Prop. 2.3]{DG_UERHG}}]\label{defn:ece}
A family $(X_j)_{j \in \I}$ of metric spaces is \emph{equi-coarsely embeddable} if for every $R, \epsilon > 0$ there is a 
family of Hilbert space valued maps $\xi_j : X_j \to \mH$ with $\| \xi_j(x) \| = 1$ for all $x \in X_j$ and satisfying:
\begin{enumerate}
	\item For all $j \in \I$ and all $x,y \in X_j$, if $d(x,y) \leq R$, then $\| \xi_j(x) - \xi_j(y) \| \leq \epsilon$.
    \item $\lim_{S\to\infty} \sup_{j \in \I} 
	 	\sup\left\{ \left| \langle \xi_j(x), \xi_j(y)\rangle \right| : d(x,y) \geq S\, , \,\, x,y \in X_j \right\} = 0$.
\end{enumerate}
\end{defn}

Now to Yu's property A.
\begin{defn}
A metric space $(X,d)$ has property A if for every $R, \epsilon > 0$ there exists an $S > 0$
and a collection $\left(A_x\right)_x$ of finite nonempty subsets of $X \times \N$, indexed by $x \in X$,
such that:
\begin{enumerate}
	\item For each $x, y \in X$, if $d(x,y) < R$, then \[ \frac{|A_x \Delta A_y|}{|A_x|} < \epsilon. \]
	\item For each $x \in X$, if $(y, n) \in A_x$ then $d(x, y) < S$.
\end{enumerate}
\end{defn}

The following characterization will be useful in the sequel.
\begin{prop}[{\cite[Prop. 3.2]{tu_RmkPropA}}]\label{prop:PropA}
A metric space $(X,d)$ has property A if and only if for every 
$R, \epsilon > 0$ there exists a Hilbert space valued map $\alpha : X \to \mH$ with each $\| \alpha_x \| = 1$, 
and an $S > 0$ such that:
\begin{enumerate}
	\item For each $x,y \in X$, if $d(x,y) \leq R$, then $\left| 1 - \langle \alpha_x , \alpha_y \rangle \right| < \epsilon$.
	\item For each $x \in X$,  if $d(x,y) \geq S$ then $\langle \alpha_x , \alpha_y \rangle = 0$
\end{enumerate}
\end{prop}

\begin{defn}\label{defn:ee}
A family $(X_j)_{j \in \I}$ of metric spaces is \emph{equi-exact} if for every $R, \epsilon > 0$ there is a family 
of Hilbert space valued maps $\xi_j : X_j \to \mH$ with $\| \xi_j(x) \| = 1$ for all $x \in X_j$ and an $S>0$ and satisfying:
\begin{enumerate}
	\item For all $j \in \I$ and all $x,y \in X_j$, if $d(x,y) \leq R$, then $\| \xi_j(x) - \xi_j(y) \| \leq \epsilon$.
	\item For all $j\in \I$ and all $x,y \in X_j$, $\langle \xi_j(x), \xi_j(y)\rangle = 0$ if $d(x,y) \geq S$.
\end{enumerate}
\end{defn}

As mentioned above, for a group extension $H \to G \to Q$, Dadarlat and Guentner show in \cite{DG_CPUE}
that if $H$ is coarsely embeddable and $Q$ has property A, then $G$ is coarsely embeddable.   The
assumption that $Q \isom G/H$ has property A is closely related to the notion of relative property A,
developed in \cite{JOR4}.  

\begin{defn}
Let $G$ be a group and $\{ H_1, \ldots, H_n \}$ a finite family of subgroups, and let $K = \sqcup_{i = 1}^{n} G/H_i$.
Then $G$ has \emph{relative property A} with respect to $\{ H_1, \ldots, H_n \}$ if for every $R>0$ and 
$\epsilon > 0$ there exists an $S > 0$ and a collection $A_x$ of finite nonempty subsets of $K \times \N$, 
indexed by $x \in G$, such that:
\begin{enumerate}
	\item For each $x, y \in G$, if $d(x,y) < R$, then \[ \frac{|A_x \Delta A_y|}{|A_x|} < \epsilon. \]
	\item For each $x \in G$, if $(aH_i, n) \in A_x$ then $d(x, aH_i) < S$.
\end{enumerate}
\end{defn}

The cosets of each single subgroup $H_i$ serve to partition the space $G$ into isometric pieces.  We obtain
a family of partitions.  A single element of one of the partitions is indexed by $K$, and distances are
measured between elements of $G$ and elements of $K$.  In the next section we expand this framework to
more general situations by considering families of set maps.

\section{Exact families of maps}
Suppose $(X,d)$ is a metric space.  For a finite family of sets, $\left\{ Y_i \right\}_{i=1}^{n}$,
let $\Y = \sqcup_{i=1}^n Y_i$ denote the disjoint union.
\begin{defn}\label{defn:exactfamily} 
A family of set maps $\left\{ \phi_i : X \to Y_i \right\}_{i=1}^{n}$ is an \emph{exact family
of maps} if for every $R, \epsilon > 0$ there exists an $S > 0$ and a map $\xi : X \to \ell^2( \Y )$,
with $\| \xi_x \| = 1$ for all $x \in X$ (where $\xi_x$ is the element of $\ell^2(\Y)$ associated to $x \in X$) 
and satisfying the following.
\begin{enumerate}
	\item For all $x,y \in X$ if $d(x,y) \leq R$, then $\| \xi_x - \xi_y \| \leq \epsilon$.
	\item For all $x \in X$, $\supp \xi_x \subset \cup_{i} \phi_i\left( B_S(x) \right)$, where $B_S(x)$ denotes the
			ball of radius $S$ in $X$ centered at $x$.
\end{enumerate}
\end{defn}

Let $G$ be a group, and let $\left\{ H_i \right\}_{i=1}^{n}$ be a finite family of subgroups of $G$.  
Let $\left\{\pi_i : G \to G/H_i\right\}$ be the corresponding quotient maps.
\begin{lem}
The group $G$ has relative property A with respect to $\left\{ H_i \right\}_{i=1}^{n}$ if and only if $\left\{ \pi_i \right\}_{i=1}^{n}$ 
is an exact family of maps.  A metric space $(X,d)$ has property A if and only if $\left\{ Id:X \to X \right\}$ is an exact family of maps.
\end{lem}
\begin{proof}
This follows from Proposition 4.3 of \cite{JOR4} and Proposition 3.2 of \cite{tu_RmkPropA}.
\end{proof}
In \cite{JOR4}, relative property A was defined only for groups.  The notion of exact families of maps gives a natural framework 
in which to extend this property to metric spaces.

\begin{thm}\label{thm:propA}
Suppose $(X,d)$ is a metric space and $\left\{ \phi_i : X \to Y_i \right\}_{i=1}^{n}$ 
is an exact family of maps.  If $\left\{ \phi_i( w )^{-1} : w \in Y_i , i=1,\ldots, n \right\}$ is an equi-exact
family of metric spaces, then $X$ has property A.
\end{thm} 
This is related to Theorem 3.1 of \cite{DG_UERHG}.
\begin{proof}
Fix $R, \epsilon > 0$.  Denote by $\Y$ the disjoint union $\sqcup_{i=1}^n Y_i$.
For $x \in X$ and $w \in Y_i$, let $\eta(x,w)$ denote a point in $\phi_i^{-1}(w)$ closest to $x$.  By the triangle inequality
we have $d(x,y) \leq d(\eta(x,w), \eta(y,w) ) + d(x, \phi_i^{-1}(w)) + d(y, \phi_i^{-1}(w))$
and $d( \eta(x,w), \eta(y, w) ) \leq d(x,y) + d(x, \phi_i^{-1}(w)) + d(y, \phi_i^{-1}(w))$ for $w \in Y_i$, and all $x,y \in X$.

By Definition \ref{defn:exactfamily} there is a map $\alpha : X \to \ell^2(\Y)$ with each $\| \alpha_x \| = 1$, 
and an $S_X > 0$ such that:
\begin{enumerate}
	\item For each $x,y \in X$, if $d(x,y) \leq R$, then 
		$\left| 1 - \langle \alpha_x , \alpha_y \rangle \right| < \frac{\epsilon}{2}$.
	\item For each $x \in X$,  if $\alpha_x( w ) \neq 0$ then $w \in \cup_i \phi_i \left( B_{S_X}(x) \right)$.
\end{enumerate}

Let $\I = \left\{ (i, w) : i=1,\ldots, n, \,\, w \in Y_i \right\}$.  Note that there is a natural identification of $\I$ with $\Y$.
For convenience, we will refer to $j = (i,w)$ as $w$.  It will be understood that $i$ is then determined by a choice of $w \in Y_i$. 
For $j = (i,w) \in \I$, set $X_j = \phi_i^{-1}(w)$.
By Definition \ref{defn:ee}, there is an $S_Y > 0$, a 
Hilbert space $\mH$, and for each $j = (i,w) \in \I$ a $\beta_j : \phi_i^{-1}(w) \to \mH$ with $\| \beta_j(s) \| = 1$ for all $s \in X_j$, and such that
\begin{enumerate}
	\item For all $s,t \in X_j$, $\left| 1 - \langle \beta_j(s), \beta_j(t) \rangle \right| \leq \frac{\epsilon}{2}$ whenever 
		$d(s,t) \leq 2S_X + R$.
	\item For all $s,t \in X_j$ if $d(s,t) \geq S_Y$, then
		$\langle \beta_j(s) , \beta_j(t) \rangle = 0$.
\end{enumerate}	

Define $\xi : X \to \ell^2( \Y, \mH )$ by
\[ \xi_x( w ) = \alpha_x( w ) \beta_j( \eta( x, w ) ), \,\,\,\, \textrm{ for all } x \in X, w \in Y_i. \]

For any $x \in X$, 
\begin{align*}
	\| \xi_x \|^2 &= \sum_{w \in \Y} \| \xi_x(w) \|^2\\
		&= \sum_{w \in \Y} | \alpha_x(w) |^2 \| \beta_j( \eta(x,w) ) \|^2 \\
		&= 1.
\end{align*}
It remains to verify the two remaining properties of Proposition \ref{prop:PropA} for $\xi$.

If $x,y \in X$ with $d(x,y) \leq R$, then
\begin{align*} |1 - \langle \xi_x , \xi_y \rangle | \leq & \left| \sum_{w \in \Y} \left( 1 - \langle \beta_j(\eta(x,w)), \beta_j(\eta(y,w)) \rangle \right) \alpha_x(w) \alpha_y(w) \right| \\
	& + | 1 - \langle \alpha_x , \alpha_y \rangle |. \end{align*}
The second term is bounded by $\frac{\epsilon}{2}$.  Moreover, the first sum is over
$w \in \phi_i ( B_{S_X}(x) ) \cap \phi_i ( B_{S_X}(y))$.  Since $\| \alpha_x \| = \| \alpha_y \| = 1$ this sum is bounded
by
\[ \sup\left\{  \left| 1 - \langle \beta_j(\eta(x,w)), \beta_j(\eta(y, w))\rangle \right| : 
	w \in \phi_i( B_{S_X}(x) ) \cap \phi_i( B_{S_X}(y)), i = 1,\ldots,n \right\}. \]
As each such $w$ satisfies $d( \eta(x,w), \eta(y,w) ) \leq R + 2 S_X$, this sum is bounded by $\frac{\epsilon}{2}$.

Let $x,y \in X$ be such that $d(x,y) \geq 2S_X + S_Y$.  Then
\[ \langle \xi_x , \xi_y \rangle = \sum_{w \in \Y} \alpha_x(w) \alpha_y(w) \langle \beta_j( \eta(x, w) ), \beta_j( \eta(y, w)) \rangle. \]
The support condition on $\alpha_x$ and $\alpha_y$ ensure the sum is over $w \in \cup_i \left(\phi_i ( B_{S_X}(x) ) \cap \phi_i ( B_{S_X}(y))\right)$.
For each $i$ and each $w \in \phi_i( B_{S_X}(x) ) \cap \phi_i( B_{S_X}(y) )$, we have  
$d( \eta(x, w), \eta(y, w) ) \geq d(x,y) - d(x, \phi_i^{-1}(w)) - d(y, \phi_i^{-1}(w)) \geq S_Y$.  As such, this sum is zero.
\end{proof}

\section{Strong embeddability}

\begin{defn}\label{defn:SE}
Let $(X,d)$ be a metric space.  Then $X$ is strongly embeddable\footnotemark if and only if for every $R, \epsilon > 0$ there
exists a Hilbert space valued map $\beta : X \to \ell^2(X)$ with $\| \beta_x \| = 1$ for each $x \in X$, and satisfying:
\begin{enumerate}
	\item[(1)] If $d(x,y) \leq R$, then $\| \beta_x - \beta_y \| \leq \epsilon$.
	\item[(2)] $\lim_{S \to \infty} \sup_{x \in X} \sum_{w \notin B_S(x)} |\beta_x(w)|^2 = 0$.
\end{enumerate}
\end{defn}

\footnotetext{It is possible to define strong embeddability in terms of a general Hilbert space.  In that framework, it is clear that strong embeddability is a coarse invariant of metric spaces.}

\begin{lem}
Condition (2) in Definition \ref{defn:SE} can be replaced by the following:
\begin{enumerate}
	\item[(2')] $\lim_{S \to \infty}	\sup_{x,y \in X} \sum_{w \notin \left(B_S(x) \cap B_S(y) \right)} |\beta_x(w)\beta_y(w) | = 0$.
\end{enumerate}
\end{lem}
\begin{proof}
	Condition (2') with $x = y$ yields condition (2).  For the other direction, assume condition (2) is satisfied,  fix any $x,y \in X$, and
	let $S > 0$ be given.
	\begin{align*}
		\sum_{w \notin \left(B_S(x) \cap B_S(y) \right)} |\beta_x(w)\beta_y(w) | &= 
				\sum_{w \notin B_S(x)} |\beta_x(w)\beta_y(w) | + \sum_{w \in \left( B_S(x) \setminus B_S(y) \right)} |\beta_x(w)\beta_y(w) | \\
			& \leq \sum_{w \notin B_S(x)} |\beta_x(w)\beta_y(w) | + \sum_{w \notin B_S(y)} |\beta_x(w)\beta_y(w) | \\
			& \leq \sum_{w \notin B_S(x)} |\beta_x(w)|^2 + \sum_{w \notin B_S(y)} |\beta_y(w) |^2
	\end{align*}
	Each of these terms tend to zero as $S \to \infty$, uniform in $x,y \in X$.	
\end{proof}

\begin{lem}
Strongly embeddable metric spaces are coarsely embeddable.
\end{lem}
\begin{proof}
Suppose $(X,d)$ is strongly embeddable, and fix $R, \epsilon > 0$.  Take $\beta : X \to \ell^2(X)$ from the definition.

Fix an $S>0$ and take $d(x,y) \geq S$.
\begin{align*}
	\left| \langle \beta_x , \beta_y \rangle \right| 
		& \leq \sum_{w \in X } \left| \beta_x(w) \beta_y(w) \right| \\
		& = \sum_{w \notin B_{S/2}(x) } \left| \beta_x(w) \beta_y(w) \right| + \sum_{w \in B_{S/2}(x) } \left| \beta_x(w) \beta_y(w) \right| \\
		& \leq \sum_{w \notin B_{S/2}(x) } \left| \beta_x(w) \right|^2 + \sum_{w \in B_{S/2}(x) } \left| \beta_y(w) \right|^2 \\
		& \leq \sum_{w \notin B_{S/2}(x) } \left| \beta_x(w) \right|^2 + \sum_{w \notin B_{S/2}(y) } \left| \beta_y(w) \right|^2
\end{align*}
Thus, $\lim_{S \to \infty} \sup \left\{ \left| \langle \beta_x , \beta_y \rangle \right| : d(x,y) \geq S \right\} = 0$.
\end{proof}

The following is a weakened version of an exact family of maps.  As an exact family of maps captures relative property A,
a strongly embeddable family of maps will capture relative coarse embeddability.
\begin{defn}\label{defn:weakexactfamily} 
A family of set maps $\left\{ \phi_i : X \to Y_i \right\}_{i=1}^{n}$ is a \emph{strongly embeddable family
of maps} if for every $R, \epsilon > 0$ there exists a map $\xi : X \to \ell^2( \Y )$,
with $\| \xi_x \| = 1$ for all $x \in X$ and satisfying the following:
\begin{enumerate}
	\item[(1)] For all $x,y \in X$ if $d(x,y) \leq R$, then $\| \xi_x - \xi_y \| \leq \epsilon$.
	\item[(2)] \[ \lim_{S\to\infty} 
	 	\sup_{x \in X}  \sum_{w \notin \cup_i \phi_i ( B_{S}(x) ) } \left| \xi_x(w) \right|^2  = 0.\]
\end{enumerate}
\end{defn}

In direct analogy to Theorem \ref{thm:propA}, we have
\begin{thm}\label{thm:ce}
Suppose $(X,d)$ is a uniformly discrete bounded geometry metric space, and let $\left\{ \phi_i : X \to Y_i \right\}_{i=1}^{n}$ 
be a strongly embeddable family of maps.  If the preimages $\left\{ \phi_i( w )^{-1} : w \in Y_i , i=1,\ldots, n \right\}$ form an equi-coarsely embeddable
family of metric spaces, then $X$ is coarsely embeddable.
\end{thm}
\begin{proof}

Pick $R, \epsilon > 0$.  By Definition \ref{defn:weakexactfamily} there is a map $\alpha : X \to \ell^2(\Y)$ with each $\| \alpha_x \| = 1$, 
and an $S_X > 0$ such that:
\begin{enumerate}
	\item For each $x,y \in X$, if $d(x,y) \leq R$, then 
		$\left| 1 - \langle \alpha_x , \alpha_y \rangle \right| < \frac{\epsilon}{3}$.
	\item For each $x,y \in X$,  $\sum_{w \notin \cup_i \phi_i ( B_{S_X}(x) ) \cap \phi_i ( B_{S_X}(y))} \left| \alpha_x(w) \alpha_y(w) \right| < \frac{\epsilon}{6}$.
\end{enumerate}

Let $\I = \left\{ (i, w) : i=1,\ldots, n, \,\, w \in Y_i \right\}$.  
For $j = (i,w) \in \I$, set $X_j = \phi_i^{-1}(w)$.
By Definition \ref{defn:ece}, there is a 
Hilbert space, $\mH$, and for each $j = (i,w) \in \I$ a $\beta_j : \phi_i^{-1}(w) \to \mH$ with $\| \beta_j(s) \| = 1$ for all $s \in X_j$, and such that
for each $\delta > 0$ there is an $S_Y > 0$ satisfying:
\begin{enumerate}
	\item For all $s,t \in X_j$, $\left| 1 - \langle \beta_j(s), \beta_j(t) \rangle \right| \leq \frac{\epsilon}{3}$ whenever 
		$d(s,t) \leq 2S_X + R$.
	\item For all $s,t \in X_j$ if $d(s,t) \geq S_Y$, then
		$\langle \beta_j(s) , \beta_j(t) \rangle \leq \frac{\delta}{2}$.
\end{enumerate}	

Define $\xi : X \to \ell^2( \Y, \mH )$ by
\[ \xi_x( w ) = \alpha_x( w ) \beta_j( \eta( x, w ) ), \,\,\,\, \textrm{ for all } x \in X, w \in Y_i. \]

For any $x \in X$, $\| \xi_x \|^2  = 1$.
We verify the two remaining properties of Proposition \ref{prop:CE} for $\xi$.

If $x,y \in X$ with $d(x,y) \leq R$, then
\begin{align*} 
|1 - \langle \xi_x , \xi_y \rangle | \leq & \left| \sum_{w \in \Y} \left( 1 - \langle \beta_j(\eta(x,w)), \beta_j(\eta(y,w)) \rangle \right) \alpha_x(w) \alpha_y(w) \right| \\
	& + | 1 - \langle \alpha_x , \alpha_y \rangle |. 
\end{align*}
The second term is bounded by $\frac{\epsilon}{3}$.  Moreover, the summation can be split as
the sum over $w \in \cup_i \phi_i( B_{S_X}(x) ) \cap \phi_i( B_{S_X}(y))$ plus the sum over $w \notin \cup_i \phi_i( B_{S_X}(x) ) \cap \phi_i( B_{S_X}(y))$.

Consider the sum over $w \in \cup_i \phi_i ( B_{S_X}(x) ) \cap \phi_i ( B_{S_X}(y))$.  Since $\| \alpha_x \| = \| \alpha_y \| = 1$ this sum is bounded by
\[ \sup\left\{  \left| 1 - \langle \beta_j(\eta(x,w)), \beta_j(\eta(y, w))\rangle \right| : 
	w \in \phi_i( B_{S_X}(x) ) \cap \phi_i( B_{S_X}(y)), i = 1,\ldots,n \right\}. \]
As every such $w$ satisfies $d( \eta(x,w), \eta(y,w) ) \leq R + 2 S_X$, this sum is bounded by $\frac{\epsilon}{3}$.

The sum over $w \notin \cup_i \phi_i ( B_{S_X}(x) ) \cap \phi_i ( B_{S_X}(y))$ is bounded by,
\[\sum_{w \notin \cup_i \phi_i ( B_{S_X}(x) ) \cap \phi_i ( B_{S_X}(y))} \left| 1 - \langle \beta_j(\eta(x,w)), \beta_j(\eta(y,w)) \rangle \right| \left|\alpha_x(w) \alpha_y(w) \right|. \]
As $\| \beta_j( \eta(x,w) ) \| = \| \beta_j ( \eta(y,w)) \| = 1$, $\left| 1 - \langle \beta_j(\eta(x,w)), \beta_j(\eta(y,w)) \rangle \right| \leq 2$.  By the choice of $\alpha$, this sum is bounded by $\frac{\epsilon}{3}$. 
Thus $| 1 - \langle \xi_x, \xi_y \rangle | \leq \epsilon$.

By Definition \ref{defn:weakexactfamily}, there is $S'_X > 0$ such that
$\sum_{w \notin \cup_i \phi_i ( B_{S'_X}(x) ) \cap \phi_i ( B_{S'_X}(y))} \left| \xi_x(w) \xi_y(w) \right| < \frac{\delta}{2}$.

Let $x,y \in X$ be such that $d(x,y) \geq 2S'_X + S_Y$.  Then
\begin{align*}
| \langle \xi_x , \xi_y \rangle | &= \left| \sum_{w \in \Y} \alpha_x(w) \alpha_y(w) 
					\langle \beta_j( \eta(x, w) ), \beta_j( \eta(y, w)) \rangle \right| \\
	&\leq \sum_{w \in \Y} |\alpha_x(w) \alpha_y(w)| 
		\left| \langle \beta_j( \eta(x, w) ), \beta_j( \eta(y, w)) \rangle \right|.
\end{align*}
This sum splits into the sum over $w \in \cup_i \phi_i( B_{S'_X}(x) ) \cap \phi_i( B_{S'_X}(y) )$ plus the sum over
$w \notin \cup_i \phi_i( B_{S'_X}(x) ) \cap \phi_i( B_{S'_X}(y) )$.

For $w \in \cup_i \phi_i( B_{S'_X}(x) ) \cap \phi_i( B_{S'_X}(y) )$, that $\|\alpha_x \| = \| \alpha_y \| = 1$ gives that this
sum is bounded by 
 \[ \sup \left\{ \left| \langle \beta_j(\eta(x, w)), \beta_j(\eta(y, w)) \rangle \right| :  
		w \in \phi_i( B_{S'_X}(x) ) \cap \phi_i( B_{S'_X}(y) ) , i=1, \ldots, n \right\}. \]
For each $i$ and each $w \in \phi_i( B_{S'_X}(x) ) \cap \phi_i( B_{S'_X}(y) )$, we have  
$d( \eta(x, w), \eta(y, w) ) \geq d(x,y) - d(x, \phi_i^{-1}(w)) - d(y, \phi_i^{-1}(w)) \geq S_Y$.  As such, this supremum is bounded 
by $\frac{\delta}{2}$.

Consider the sum over $w \notin \cup_i \phi_i( B_{S'_X}(x) ) \cap \phi_i( B_{S'_X}(y) )$.  As 
$\| \beta_j( \eta(x,w) ) \| = \| \beta_j( \eta(y,w) \| = 1$, $|\langle \beta_j( \eta(x,w) ) , \beta_j( \eta(y,w) ) \rangle | \leq 1$.
By the choice of $S'_X$, this sum is bounded by $\frac{\delta}{2}$.  Thus, $| \langle \xi_x , \xi_y \rangle | < \delta$.
\end{proof}

\begin{defn}
A family $(X_j)_{j \in \I}$ of metric spaces is \emph{equi-strongly embeddable} if for every $R, \epsilon > 0$ there is a 
family of Hilbert space valued maps $\xi^j : X_j \to \ell^2( X_j )$ with $\| \xi^j_x \| = 1$ for all $x \in X_j$ and satisfying:
\begin{enumerate}
	\item For all $j \in \I$ and all $x,y \in X_j$, if $d(x,y) \leq R$, then $\| \xi^j_x - \xi^j_y \| \leq \epsilon$.
    \item $\lim_{S\to\infty} \sup_{j \in \I} 
	 	\sup_{x \in X_j} \sum_{w \notin B_S(x)} | \xi^j_x(w) |^2 = 0$.
\end{enumerate}
\end{defn}

\begin{cor}\label{cor:stronglyembeddable}
Suppose $(X,d)$ is a uniformly discrete bounded geometry metric space, and let $\left\{ \phi_i : X \to Y_i \right\}_{i=1}^{n}$ 
be a strongly embeddable family of maps.  If the preimages $\left\{ \phi_i( w )^{-1} : w \in Y_i , i=1,\ldots, n \right\}$ are an 
equi-strongly embeddable family of metric spaces, then $X$ is strongly embeddable.
\end{cor}
\begin{proof}
This follows directly by the argument used in Theorem \ref{thm:ce} above.
\end{proof}

Since a strongly embeddable family of maps is a generalization of relative property A, we have the following.
\begin{cor}
If $G$ has relative property A with respect to a finite family of subgroups $H_1, \ldots, H_n$, and if each $H_i$ is
strongly embeddable, then $G$ is strongly embeddable.
\end{cor}

In \cite{JOR4} it was shown that if a group $G$ has a normal subgroup $H$ with $G/H$ of property A, then $G$ has relative
property A with respect to $H$.  As such, the following is a generalization of the previously mentioned Dadarlat and Guentner
result on the coarse embeddability of extensions.
\begin{cor}
Suppose $G$ has relative property A with respect to a finite family of subgroups $\{H_i\}_{i = 1}^{n}$.  If each $H_i$ is 
coarsely embeddable, then $G$ is coarsely embeddable.
\end{cor}
\begin{proof}
If each $H_i$ is coarsely embeddable, then $\left\{ \pi_i^{-1}( aH_i ) : aH_i \in G/H_i\, , \,\, i=1,\ldots,n \right\}$ is
an equi-coarsely embeddable family.  As relative property A is equivalent to $\left\{ \pi_i : G \to G/H_i \right\}_{i=1}^{n}$
being an exact family of maps, this follows from Theorem \ref{thm:ce}.
\end{proof}

We next turn to group actions on metric spaces.  Suppose the group $G$ acts co-finitely on a uniformly 
discrete bounded geometry metric space $X$.  Let$x_1, x_2, \ldots, x_n \in X$ be representatives of the $G$ orbits.  For each 
$i$, let $H_i$ be the stabilizer of $x_i$, and denote by $\pi_i : G \to G/H_i$ the quotient map.
For each $x_i$, there is a natural identification of the orbit $G x_i$ with $G/H_i$, thus between
$X$ and $\sqcup_{i=1}^{n} G/H_i$.  Thus if $X$ is a finite space, then each $H_i$ has finite index in $G$.
The following then follows from Proposition 4.11 of \cite{JOR4}.

\begin{thm}\label{finiteIndex}
If $X$ is finite, then $\left\{ \pi_i : G \to G/H_i \right\}$ is an exact family of maps.
\end{thm}

For groups acting on infinite metric spaces, the following lemma will be useful.

\begin{lem}\label{lem:TechHyp}
There is a sequence $(T_k)$ of positive numbers, tending to infinity, such that for all $w \in X$, if $d_X( x_1, w ) < T_k$ then there
is $g_w \in G$ and $i$ with $w = g_w x_i$, such that $d_G(1_G, g_w) < k$.  That is, 
$B_{T_k}( gx_1) \subset \cup_{i=1}^{n} B_k(g) x_i$ for all $g \in G$.
\end{lem}
\begin{proof}
For each $w \in X$, fix a factorization $w = g_w x_{i_w}$, for $x_{i_w}$ one of the orbit representatives, and
$g_w \in G$.  There may be many possible choices for $g_w$.  Take one which minimizes $d_G( 1_G, g_w)$.

For a positive integer $m$, set $N_m = \max \left\{ d_G( 1_G, g_w ) \, : \, w \in B_m(x_1) \right\}$.
Clearly $N_m \leq N_{m+1}$ for all positive integers $m$.

If $N_m \leq S$ for all $m$, then
each $w \in X$ could be written as $g_w x_{i_w}$ with $d_G( 1_G, g_w) \leq S$.  As $G$ has bounded
geometry, there are at most finitely many possibilities for $g_w$, and finitely many orbits.  This contradicts
the assumption that $X$ has infinite cardinality.  Thus $N_m$ tends to infinity.

Using $(N_m)$ we now construct $(T_n)$.  Given $n$, let $T_n = \max \{ m \, : \, N_m \leq n \}$.
If  $d_X(w, x_i ) < T_n$, then $d_G( 1_G, g_w ) \leq N_{T_{n}} \leq n$.
\end{proof}

The following is a strongly embeddable analogue of Theorem 5.12 of \cite{JOR4}.
\begin{thm}\label{thm:actionse}
Suppose the group $G$ acts co-finitely on a strongly embeddable uniformly discrete bounded geometry metric space $X$.  Let
$x_1, x_2, \ldots, x_n \in X$ be representatives of the $G$ orbits.  For each $i$, let $H_i$ be the stabilizer of $x_i$, and 
denote by $\pi_i : G \to G/H_i$ the quotient map.  Then $\left\{ \pi_i \, : \, i = 1, \ldots, n \right\}$ is a strongly embeddable family of maps.
\end{thm}
\begin{proof}

If $X$ is finite, this follows by Theorem \ref{finiteIndex}.  Otherwise, we assume $X$ is infinite.
Fix $R, \epsilon > 0$.   There is a $C > 0$ such that for all $g, g' \in G$, and all $i, j \in \{ 1, \ldots, n\}$,
$d_X( g x_i, g' x_j ) \leq C( d_G( g, g' ) + 1 )$.

Since $X$ is strongly embeddable, there is a Hilbert space valued map $\beta : X \to \ell^2(X)$ with $\| \beta_x \| = 1$ 
for each $x \in X$, and satisfying:
\begin{enumerate}
	\item[(1)] If $d_X(x,y) \leq C(R+1)$, then $\| \beta_x - \beta_{y} \| \leq \epsilon$.
	\item[(2)] $\lim_{S \to \infty} \sup_{x \in X} \sum_{w \notin B_S(x)} |\beta_x(w)|^2 = 0$.
\end{enumerate}

Define $\xi : G \to \ell^2(X)$ by  $\xi_g = \beta_{g x_1}$.
It is clear that for all $g \in G$, $\| \xi_g \| = 1$.
If $d_G(g,g') \leq R$, then $d_X(g x_1, g' x_1 ) \leq C(R+1)$ so $\| \xi_g - \xi_{g'} \| \leq \epsilon$.

For any $g \in G$ and $S > 0$
\begin{align*}
	\sum_{w \notin \cup_i  B_{S}(g) x_i } \left| \xi_g(w) \right|^2 &=
			\sum_{w \notin \cup_i  B_{S}(g) x_i } \left| \beta_{gx_1}(w) \right|^2 \\
		&\leq \sum_{w \notin B_{T_S}(gx_1)  } \left| \beta_{gx_1}(w) \right|^2,
\end{align*}
where the last inequality follows from Lemma \ref{lem:TechHyp}.
This tends to zero uniformly in $g$, as $S \to \infty$.

\end{proof}

\begin{thm}
Let $H$ be a normal subgroup of $G$.  If $G/H$ is strongly embeddable
in the quotient metric, then $\left\{ \pi : G \to G/H \right\}$ is a strongly embeddable family of maps.
\end{thm}
\begin{proof}
This follows from Theorem \ref{thm:actionse} by considering the $G$ action on $G/H$.
\end{proof}

If $H$ is strongly embeddable as well, the collection of cosets $\{ gH \}$ forms an equi-strongly embeddable family
of metric spaces.  Appealing to Corollary \ref{cor:stronglyembeddable} we obtain the following.
\begin{thm}
The collection of strongly embeddable groups is closed under extensions of groups with proper length functions.
\end{thm}

\bibliography{all}

\end{document}